\documentclass[11pt,a4paper]{article}
\usepackage{amssymb,amsmath,amsthm,bm}
\usepackage{graphicx}
\usepackage{geometry}
\usepackage{float}
\usepackage{tikz}
\graphicspath{{./figures/}}
\usepackage[maxbibnames=99]{biblatex}
\newtheorem{teo}            {Theorem}
\newtheorem{lema}     [teo]{Lemma}
\newtheorem{defin}[teo]{Definition}
\newtheorem{corollary} [teo]{Corollary}

\newtheorem{example} [teo]          {Example}
\newtheorem{obs} [teo]           {Remark}
\newtheorem{prop} [teo]       {Proposition}

\DeclareMathOperator{\ric}{Ric}
\DeclareMathOperator{\Ric}{Ric}

\title{Ricci curvature, Bruhat graphs and Coxeter groups}
\author{Viola Siconolfi\ \footnotemark[1]}

\begin{document}
\maketitle
\footnotetext[1]{
Dipartimento di Matematica, Universita di Roma 'Tor Vergata', Via della Ricerca Scientifica, 00133 Rome, Italy.
\textit{Email address: }\texttt{siconolf@mat.uniroma2.it}.} 

\footnotetext[2]{
An extended abstract containing some of the results of this paper will appear in the proceedings of FPSAC 2020 \cite{sico}}

\begin{abstract}
We consider the notion of discrete Ricci curvature for graphs defined by Schmuckenschl{\"a}ger \cite{shmuck} and compute its value for Bruhat graphs associated to finite Coxeter groups. To do so we work with the geometric realization of a finite Coxeter group and a classical result obtained by Dyer in \cite{Dyer}. As an application we obtain a bound for the spectral gap of the Bruhat graph of any finite Coxeter group and an isoperimetric inequality for them. Our proofs are case-free.
\end{abstract}


\section{Introduction}

In \cite{klart} the authors studied the discrete Ricci curvature of various Cayley graphs associated to different groups. 
Among these graphs they worked on the Ricci curvature of the Bruhat graph of type $A$. 
The definition of discrete curvature used was the one  by Schmuckenschl{\"a}ger in \cite{shmuck}. This was defined as an extension of a notion of curvature defined for operators on probability spaces in \cite{bakem}.

The study of Ricci curvature for graphs is part of a trend in graph theory of translating notions of Riemannian geometry to graphs, we have for example the Laplacian (\cite{bre3}) and the Harnack inequality (\cite{bre4},\cite{bre9} and \cite{bre16}) for graphs.
In their paper \cite{klart}  Klartag et al. proved some interesting applications of the Ricci curvature: knowing the Ricci curvature of a graph it is possible to have information about its spectral gap (see definition  \ref{spec}) and obtain an isoperimetric inequality, the latter in analogy with Riemannian geometry.

Part of the
difficulty with the study of the Ricci curvature from Schmuckenshl{\"a}ger is the lack of examples where explicit computations
are possible. 
In this paper we compute the Ricci curvature of the Bruhat graph of  finite Coxeter groups ( Theorem \ref{B(W)}). 
The proof relies on geometric aspects of Coxeter groups, in particular on a result of Dyer \cite{Dyer}. As corollaries of this main result we obtain a bound for the spectral gap of Bruhat graphs and an isoperimetric inequality for them.

The paper is divided as follows.
In Section \ref{prelimin} we recall the necessary preliminaries of both Coxeter theory and discrete Ricci curvature. 
In particular we review the structure of the Bruhat graph and the geometric representation of a finite Coxeter system. The definition of Ricci curvature is then presented and further some results from \cite{bakem} are recalled. 

 Section \ref{secbru} is the one of our main result. We start by defining an equivalence relation on the vertices of a Bruhat graph whose distance from the identity is $2$. This allows us to compute the Ricci curvature of the Bruhat graph of any finite Coxeter group from the corresponding analysis on dihedral groups.
 
 Some applications are presented in the last Section.

\section{Preliminaries}\label{prelimin}
\subsection{ Coxeter groups}\label{ssec1}

This subsection is devoted to recalling classical results of Coxeter group theory that are needed to prove our main result. The main references are \cite{BB} and \cite{hump}.

We define Coxeter systems as follows:

\begin{defin}
Let $W$ be a group generated by the elements in $S$	 and $S=\{s_i\}_{i\in I}$ be a finite set with the following relations:
\[
(s_is_j)^{m_{ij}}=e;
\] 
with $m_{ii}=1$ for all $i\in I$ and $m_{ij}=m_{ji}\geq 2$ (including $m_{ij}=\infty$) for all $i,j\in I, i\neq j$. Then
$W$ is called a Coxeter group, $S$ turns out to be a  minimal set of generators, its elements are called Coxeter generators.
\end{defin}

The values $m_{ij}$ are usually seen as the coefficients of a symmetric matrix .  

All the information about a Coxeter group can be encoded in a graph called Coxeter graph. Its set of vertices is $S$ and there is an edge between two vertices $s_i$ and $s_j$ if $m_{ij}\geq 3$, such an edge is labelled with $m_{ij}$ if $m_{ij}\geq 4$.

\begin{defin}
We call length of an element $w$ in $W$ the minimal length $\ell(w)$ of all the possible expressions of $w$ as product of element in $S$
\end{defin}

 We define in $W$ the set of reflections as the union of all the conjugates of $S$:
 $$T:=\cup_{w\in W}wSw^{-1}.$$ 
 The definitions of length and reflections allow us to define a partial order on the set $W$:

\begin{defin}
Given $w\in W$ and $t\in T$, if $w'=tw$ and
$\ell(w')<\ell(w)$ we write $w\leftarrow w'$.
 Given two elements $v,w\in W$ we say that 
 $v \geq w$ according to the Bruhat order if there are 
 $w_0,\ldots, w_k\in W$ such that
\[
v=w_0\leftarrow w_1\ldots w_{k-1}\leftarrow w_k=w.
\]
\end{defin}
 The Bruhat graph associated to a Coxeter group, denoted $B(W)$, is the graph where the set of vertices is the set of elements in $W$ and there is an edge between two vertices $w,v$ if and only if either $w\leftarrow v$ or $v\leftarrow w$, we consider all these edges to be undirected.

Coxeter groups were introduced as a generalization of groups generated by reflections. The results that follow are linked to this and have a more geometrical character.

Given a finite Coxeter system $(W,S)$, it is known (see e.g. \cite[Subsection 5.3]{hump}) that it can be represented as a group generated by orthogonal reflections.
More precisely, we take $V$ as a real vector space of dimension $|S|$ with basis $\Pi$. We fix a bijection between $\Pi$ and $S$ that allows us to define the index $m_{\alpha,\beta}$ for pairs $\alpha,\beta \in \Pi$ . Furthermore we define $\langle,\rangle$ a positive definite, symmetric, bilinear form defined on $V$ that acts on the basis as follows:
\[
\langle\alpha,\beta \rangle=-\cos\left(\frac{\pi}{m(\alpha,\beta)}\right);\quad \alpha,\beta\in \Pi.
\]
Notice that $\langle\alpha,\alpha\rangle=1$ for every $\alpha \in \Pi$.
For any $\alpha$ in $V\setminus\{0\}$ we consider the reflection $r_{\alpha}$ in $V$ that acts as follows:
\[
r_{\alpha}(x)=x-\frac{2\langle x,\alpha\rangle}{\langle\alpha,\alpha\rangle}\alpha.
\]
Let 
$\tilde{W}$ be the group generated by $\{r_{\alpha}|\alpha\in \Pi\}$, then $W\equiv \tilde{W}$.
We let $\Phi=W\Pi$ be the set of roots associated to this representation of $W$ and $\Phi^+=\{\alpha\in\Phi|\alpha=\sum_{\gamma\in\Pi}c_{\alpha\gamma}\gamma, c_{\alpha\gamma}\geq 0\}$ be the subset of positive roots.
The reflections of $W$ are in bijection with $\Phi^{+}$ by the correspondence $\alpha\rightarrow r_{\alpha}$.

We conclude by listing some facts that are used in Section \ref{secbru} .

\begin{defin}
A Coxeter group $(W,S)$ is called dihedral if $|S|=2$ and $m_{1,2}\geq 3$ where $S=\{s_1,s_2\}$.
\end{defin}

.
\begin{lema}\label{dyer}
Suppose $t_1,t_2,t_3,t_4$ are reflections in $W$ and $t_1t_2=t_3t_4\neq e$. Then the subgroup $W'=<t_1,t_2,t_3,t_4>\subset W$ is a dihedral Coxeter group.
\end{lema}

\begin{corollary}\label{corody}
Given $\alpha_1$ and $\alpha_2$ two positive roots in $\Phi$, we define the following subgroup:
\[
W':=<r_{\alpha}|\alpha\in(\mathbb{R}\alpha_1\oplus\mathbb{R}\alpha_2)\cap\Phi^+>.
\]
Then $W'$ is dihedral and is the largest dihedral subgroup of $W$ containing $r_{\alpha_1}$ and $r_{\alpha_2}$.
\end{corollary}

A proof of Corollary \ref{corody} and Lemma \ref{dyer} can be found in \cite[Section 3]{Dyer}, see also \cite{dyer2}.

\subsection{Ricci curvature of a locally finite graph}\label{radici}
We recall the discrete Ricci curvature of a locally finite graph, the main reference for this section is \cite[Subsection 1.1]{klart}, the notation is slightly different.  

Let $G$ be a graph, 
we assume it to be undirected, with no loops and with no multiple edge (i.e. a simple graph); also we ask that it has no isolated vertices.
We denote by $\mathcal{V}(G)$ the set of vertices of $G$, by $\mathcal{E}(G)$ its edges and by $\delta(x,y)$ the function $\delta:\mathcal{V}(G)\times\mathcal{V}(G)\rightarrow \mathbb{N}\cup \{\infty\}$ that gives the distance between two vertices. For a fixed $x\in\mathcal{V}(G)$  and $i\in\mathbb{N}$ we define the set:
\[
B(i,x):= \{u\in\mathcal{V}(G)|\delta(x,u)=i\}; 
\]
by $d(x)$ we denote the cardinality of $B(1,x)$ and call it the degree of $x$. 

\begin{defin}
We say that a given graph $G$ is locally finite if $d(x)< \infty$ for all $x\in \mathcal{V}(G)$.
\end{defin}
From now on we assume $G$ to be locally finite.

Given real functions $f$ and $g$ on $\mathcal{V}(G)$ and $x\in\mathcal{V}(G)$ we define the following operators:
\begin{itemize}
  \item $\Delta (f)(x):=\sum_{v\in B(1,x)}(f(v)-f(x))$;
  \item $\Gamma (f,g)(x):=\frac{1}{2}\sum_{v\in B(1,x)}(f(x)-f(v))(g(x)-g(v))$;
  \item $\Gamma_2(f)(x):=\frac{1}{2}\left(\Delta\Gamma(f,f)(x)\right)-\Gamma(f,\Delta(f))(x)$.
\end{itemize}
Instead of $\Gamma(f,f)(x)$ we write $\Gamma(f)(x)$. Note that $\Gamma(f)(x)\geq 0$ $\forall x\in \mathcal{V}(G)$, the equality holds if and only if $f(x)=f(v)$ for all $v$ in $B(1,x)$. Also note that $\Delta(f)(x)=\sum_{v\in B(1,x)}f(v)-d(x)f(x)$.

\medskip
These definitions allow us, following \cite{klart} and \cite{shmuck}, to introduce the Ricci curvature of a graph:
\begin{defin}\label{curvy}
The discrete Ricci curvature of a graph $G$, denoted $\ric(G)$, is the maximum value $K\in\mathbb{R}\cup \{-\infty\}$ such that for any real function $f$ on $\mathcal{V}(G)$ and any vertex $x$, $\Gamma_2(f)(x)\geq K\Gamma(f)(x)$ .
\end{defin}
There is also a local version of this definition: 
\begin{defin}
The local Ricci curvature of a graph $G$ at a given point $\bar{x}\in\mathcal{V}(G)$ is the maximum value $K\in\mathbb{R}\cup \{-\infty\}$ such that for any real function $f$ on $\mathcal{V}(G)$ , $\Gamma_2(f)(\bar{x})\geq K\Gamma(f)(\bar{x})$ holds. The local curvature so defined is denoted $\Ric(G)_{\bar{x}}$. 
\end{defin}

We obtain that $\Ric(G)=inf_{x\in\mathcal{V}(G)}\Ric(G)_x$.

For brevity, in the rest of this work, we often say "curvature" instead of "Discrete Ricci curvature".

\begin{obs}
Note that if $c,d\in\mathbb{R}$
and $f:\mathcal{V}(G)\rightarrow \mathbb{R}$
then $\Delta(f)=\Delta(f+c)$, $\Gamma(f+c,g+d)=\Gamma(f,g)$ and $\Gamma_{2}(f+c)=\Gamma_2(f)$. 
We may therefore assume that in the expression $\Gamma_2(f)(x)\geq K\Gamma(f)(x)$, $f$ satisfies $f(x)=0$.
\end{obs}

The following result appears in \cite[Section 1.1]{klart}:

\begin{prop}\label{proof2}
$\Gamma$  and $\Gamma_2$ can be expressed through the following formulas when $f(x)=0$:
$$\Gamma (f)(x)=\frac{1}{2}\sum_{v\in B(1,x)}f(v)^2, $$
$$2\Gamma_2(f)(x)=\frac{1}{2}\sum_{u\in B(2,x)}\sum_{v\in B(1,u)\cap B(1,x)}(f(u)-2f(v))^2+\left( \sum_{v\in B(1,x)}f(v) \right)^2+
$$
$$
+\sum_{\{v,v'\}\in \mathcal{E}(G)}\left(2(f(v)-f(v'))^2+\frac{1}{2}(f(v)^2+f(v')^2)\right)+\sum_{v\in B(1,x)}\frac{4-d(x)-d(v)}{2} f(v)^2,
$$
where the third sum runs over all $v,v'\in B(1,x)$ such that $\{v,v'\}$ is an edge in $G$.
\end{prop}

The following simple observation does not appear anywhere in the literature so we include its proof:

\begin{lema}\label{proof3}
We can write
  \begin{equation}\label{equcurv}
  \text{Ric}(G)=\text{inf}_{x,f} \frac{\Gamma_2(f)(x)}{\Gamma(f)(x)} 
  \end{equation}
  where $x$ ranges in $\mathcal{V}(G)$ and $f$ ranges over the real functions defined on $\mathcal{V}(G)$ such that $f(x)=0$ and $\Gamma(f)(x)>0$.
\end{lema}

\begin{proof}
We know that $G$ has no isolated vertices, then for any $\bar{x}\in \mathcal{V}(G)$ there is a function $\bar{f}$ on $\mathcal{V}(G)$ such that $\Gamma(\bar{f})(\bar{x})> 0$. Indeed it is sufficient that $\bar{f}$ does not vanish on $B(1,\bar{x})$ which is a non-empty set. Any $K$ that satisfies 
$$\Gamma_2(f)(x)\geq K\Gamma(f)(x )\quad \forall x,\forall f$$  also satisfies
$$
K\leq \frac{\Gamma_2(g)(x)}{\Gamma(g)(x)}<\infty
\quad \forall x,g \text{ s.t.} \Gamma(g)(x)>0.
$$
Let now $\tilde{x}\in \mathcal{V}(G)$ and $\tilde{f}:\mathcal{V}(G)\rightarrow \mathbb{R}$ be a pair such that $\Gamma(\tilde{f})(\tilde{x})=0$, it easily comes from the definitions that in this case $\Gamma_2(\tilde{f})(\tilde{x})\geq 0$ and so in particular $\Gamma_2(\tilde{f})(\tilde{x})\geq K\Gamma(\tilde{f})(\tilde{x})$.
We conclude that the Ricci curvature verifies
 $\Ric(G)\leq\frac{\Gamma_2(\bar{f})(\bar{x})}{\Gamma(\bar{f})(\bar{x})}< \infty$, and from this the statemant follows. 
\end{proof}

\begin{obs}
From Lemma \ref{proof3} we deduce the following formula for the local Ricci curvature:
\[
  \text{Ric}(G)_x:=\text{inf}_{f} \frac{\Gamma_2(f)(x)}{\Gamma(f)(x)}.
\]
Where $f$ ranges among the functions on $\mathcal{V}(G)$ such that $f(x)=0$ and $\Gamma(f)(x)\neq 0$.
\end{obs}

\begin{obs}\label{obs2}
From the formulas of $\Gamma$ and $\Gamma_2$ we obtain that $\Ric(G)_x$ depends only on the subgraph obtained as the union of the paths of length 1 and 2 starting from $x$. We will refer to such a subgraph as the length-2 path subgraph of $x$. As a consequence two vertices with isomorphic length-2 path subgraphs have the same local Ricci curvature.
\end{obs}

We conclude with the following remark about triangle-free graphs. These are graphs with no $x,u,v\in\mathcal{V}(G)$ such that $\{x,u\},\{x,v\},\{u,v\}\in \mathcal{E}(G)$:

\begin{obs}\label{triang}
If $G$ is a triangle free graph, then:
\[
2\Gamma_2(f)(x)=\frac{1}{2}\sum_{u\in B(2,x)}\sum_{v\in B(1,u)\cap B(1,x)}(f(u)-2f(v))^2+
\]
\[
+(\sum_{v\in B(1,x)}f(v))^2+\sum_{v\in B(1,x)}\frac{4-d(x)-d(v)}{2} f(v)^2.
\]
We obtain this formula from Proposition \ref{proof2}.
\end{obs}

We end this subsection by recalling \cite[Theorem 1.2]{klart} and a corollary:

\begin{teo}\label{teotriang}
Let $G$ be a locally finite graph, $t(v,v')$ be the function that counts the number of triangles containing both the vertices $v$ and $v'$ and let $T=sup_{\{v,v'\}\subset \mathcal{V}(G)} t(v,v')$. Then $\Ric(G)\leq 2+\frac{T}{2}$.
\end{teo}

\begin{corollary}\label{coro}
Let $G$ be a graph with no triangles, then $\Ric(G)\leq2$.
\end{corollary}
.

Knowing the Ricci curvature of a graph allows one to obtain other information about it. We now recall some results about the spectral gap that depend on a lower bound on the Ricci curvature. Our references are \cite[Section 3]{klart}, \cite[Section 4]{klart}. To learn more about isoperimetric inequalities for graphs we address the reader to \cite[Section 4]{isop}. The graphs considered are finite.

 We start by recalling the definition of the spectral gap

\begin{defin}
We define the adjacency matrix of $G$ as a square matrix of dimension $|\mathcal{V}(G)|\times|\mathcal{V}(G)|$ such that the entry $A_{ij}$ of the matrix is $1$ if there is an edge between the vertex $i$ and the vertex $j$ of $G$, $A_{ij}$ is $0$ otherwise.

Then the Laplacian matrix is defined as:
\[
\overline{\Delta}:=-D+A
\]
where $D$ is the diagonal matrix whose entries are the degrees of the vertices and $A$ is the adjacency matrix of the graph.
\end{defin}

\begin{defin}\label{spec}
Let $G$ be a graph and $\overline{\Delta}$ its Laplacian matrix, we define the spectral gap of $G$, $\lambda_G$ as the least non zero eigenvalue of $-\overline{\Delta}$.
\end{defin}
 Where no ambiguity occurs we shall just denote the spectral gap as $\lambda$.
The next results, proved in \cite{klart}, show how the Ricci curvature of a graph can help bound its spectral gap and how these two values imply an isoperimetric inequality.

\begin{teo}\label{spgap}
Let $G$ be a graph with $\Ric(G)>0$, then $\lambda\geq \Ric(G)$.
\end{teo}

\begin{teo}\label{sgiso}
Suppose $G$ has $\Ric(G)\geq K$, for some $K\in \mathbb{R}\setminus \{0\}$. Then, for any subset $A\subset \mathcal{V}(G)$,
\[
|\partial A| \geq \frac{1}{2}min \{\sqrt{\lambda},\frac{\lambda}{\sqrt{2|K|}}\}|A|\left(1-\frac{|A|}{|\mathcal{V}(G)|}\right) .
\]
Here, by $\partial A$, we mean the collection of all edges connecting $A$ to its complement.
\end{teo}

Notice that for Theorem \ref{spgap} it is crucial to have $\Ric(G)$ positive, but positivity is not required for Theorem \ref{sgiso}. Given any bound on the spectral gap it is still possible to apply Theorem \ref{sgiso} to families of graphs with a negative lower bound for the Ricci curvature.

\section{Ricci curvature of Bruhat graphs}\label{secbru}

In this Section  
we work with graphs associated to Coxeter groups, namely we compute the discrete Ricci curvature of Bruhat graphs. 
We obtain as a Corollary a lower bound on the spectral gap of these graphs and an isoperimetric inequality. We recall that according to our notation fixed in Subsection \ref{ssec1} the Bruhat graph associated to a Coxeter group $W$ is denoted $B(W)$. The Discrete Ricci curvature is defined only for locally finite graphs, therefore during this whole chapter all the Coxeter groups are meant to be finite.

The theorem we prove is the following:

\begin{teo}\label{B(W)}
Given a finite Coxeter system, the discrete curvature of its Bruhat graph is $2$.
\end{teo}

For type $A_n$ Coxeter groups this result is proved in  \cite[Section 2.6]{klart}, notice that our result holds for any finite Coxeter group and not only for the irreducible ones.

First we notice that graph automorphisms respect the Ricci curvature: 
\begin{lema}\label{stesso}
If $G$ is any locally finite graph and $\chi$ is a graph automorphism then the following holds:
$$
Ric(G)_x=Ric(G)_{\chi(x)}.
$$
\end{lema}
\begin{proof}
Trivial from Remark \ref{obs2}.
\end{proof}

 We start studying the neighbours of vertices in $B(W)$ for a given Coxeter group $W$. Note that for $g\in W$, $B(1,g)=\{sg| s$ reflection in $ W\}$ and $B(2,g)=\{ss'g|s\neq s'$ reflections in $W\}$.
 Our first step toward the proof of Theorem \ref{B(W)} is the definition of an equivalence relation on the elements of $B(2,e)$, $e$ being the identity of $W$. We therefore take the two sets $B(2,e)$ and $B(1,e)$. For any $u\in B(2,e)$ we define the subgroup $G_u\subset W$ as follows:

\[
G_u:=<s\in T|\exists t\in T \text{ s.t. } st=u \text{ or } ts=u >,
\]

i.e., $G_u\subset W$ is the subgroup generated by all the reflections $s_i,s_j$ such that $s_is_j=u$.
We now define an equivalence relation on $B(2,e)$: given $u,u'\in B(2,e)$, we say that $u\sim u'$ if and only if $G_u=G_{u'}$.

\begin{example}
As an example of the group $G_u$ we consider $W=B_4$ and we see it as a group generated by reflections in $\mathbb{R}^4$. The reflections in $W$ are of the following types:
\begin{itemize}
\item $(i,-i)$ for $i=1,2,3,4$ that acts on $\mathbb{R}^4$ by changing the sign of the $i$-th coordinate;
\item $(i,j)(-i,-j)$ for $i,j=1,2,3,4$, $i\neq j$ that acts on $\mathbb{R}^{4}$ by switching the $i$-th and the $j$-th coordinates;
\item $(i,-j)(-i,j)$ for $i,j=1,2,3,4$, $i\neq j$ that acts on $\mathbb{R}^{4}$ by inverting the sign of the coordinates $i$ and $j$ and then switching them.
\end{itemize}
We consider $u=[-1,-2,3,4]\in B(2,e)$. We want to compute $G_u$ and find  the $u'\in B(2,e)$ such that $u\sim u'$.

First we find all the pairs of reflections $t,t'$ such that $tt'=u$.
Notice that
$u^2=e$, so $t$ and $t'$ must commute,
then notice that both $t$ and $t'$ must act on at least one among the first two coordinates. So the only possible pairs $(t,t')$ are $((1,-1),(2,-2))$, $((2,-2),(1,-1))$, $((1,2)(-1,-2),(1,-2)(-1,2))$ and $((1,-2)(-1,2),(1,2)(-1,-2))$. It follows that 
\[
G_u=<(1,-1),(2,-2),(1,-2)(-1,2),(1,2)(-1,-2)>.
\]
In this case $G_u$ is generated by $s_1=(1,-1)$ and $s_2=(1,-2)(-1,2)$.
The relations on $s_1$ and $s_2$ are $s_1^2=s_2^2=e$ and $(s_1s_2)^2=e$, so $G_u$ is isomorphic to the dihedral group $I_2(2)$  
and $u=(1,-1)(2,-2)$ is a rotation in it. Any other $u'$ in the equivalence class of $u$ is of type $\rho$, where $\rho$ is a non trivial rotation in $G_u$. We conclude that the equivalence class of $u$ in $B(2,e)$ is
\[
\{(1,-1)(2,-2), (1,-1)(1,2), (2,-2)(1,-2)\}.
\]
Notice that for two elements $u,u'\in B(2,e)$ to be in the same equivalence class according to $\sim$ we must have $G_u=G_{u'}$, it is not sufficient to have $G_u\sim G_{u'}$. In particular, taking $u'=[1,2,-3,-4]$ it is easy to see that $G_u\sim G_u'$, but because $G_u\neq G_{u'}$ we have $u\not\sim u'$.
\end{example} 

Now we focus on the structure of the subgroups $G_u$:
\begin{prop}\label{prop1}
For any $u\in B(2,e)$, $G_u$ is a dihedral reflection subgroup. Furthermore, for any pair of reflections $t_1,t_2$ such that $t_1t_2=u$, $G_u$ is the maximal dihedral reflection subgroup that contains $\{t_1,t_2\}$.
\end{prop}

\begin{proof}
Let 
$\alpha_i$ be the positive root associated to $t_i$ through the bijection $\alpha\rightarrow r_{\alpha}$.
We consider the subgroup described in Corollary \ref{corody}: $W'_{12}=<r_{\alpha}|\alpha\in\mathbb{R}\alpha_1\oplus \mathbb{R}\alpha_2\cap\Phi^+>$, this is the maximal dihedral reflection subgroup that contains $t_1$ and $t_2$, we show that $G_u=W'_{12}$.
Let $t_3\in G_u$ then there exists $t_4\in T$ such that $t_3t_4=u$. So
we have that $t_1t_2=t_{3}t_{4}\neq 1$ so we can apply Lemma \ref{dyer} and say  that $<r_1,r_2,r_3,r_4>$ is a dihedral group. It follows that the roots $\alpha_1,\alpha_2,\alpha_3$ and $\alpha_4$ lie on the same plane $\pi\subset \mathbb{R}^4$. Because $\alpha_i$ for $i=1,\ldots,4$ are pairwise independent we can write $\pi=\mathbb{R}\alpha_1\oplus \mathbb{R}\alpha_2$.

 So $t_3\in W'_{12}$ and it follows that $G_u\subset W_{12}'$. Conversely notice that $u\in W_{12}'$ and because it is the product of two reflections $u$ must be a rotation. For any rotation $\rho$ and any reflection $t$ in a dihedral group isomorphic to $I_2(m)$ there is a reflection $t'$ such that $tt'=\rho$, indeed we only need to take $k\equiv i-j $ mod $m$.
Therefore for any reflection $s$ in $W_{12}'$, there exists a $s'$ such that $ss'=u$. We conclude that $s\in G_u$ so $W_{12}'\subset G_u$.
\end{proof}

\begin{corollary}\label{coropair}
Let $u,u'\in B(2,e)$, if $G_u$ and $G_{u'}$ have two or more reflections in common then $G_u=G_{u'}$.
\end{corollary}
\begin{proof}
Let $s_{\alpha},s_{\beta}\in G_u\cap G_{u'}$, where
$s_{\alpha},s_{\beta}$  are reflections associated to the positive roots $\alpha$ and $\beta$. We notice that 
$G_u=W_{\alpha\beta}=<r_{\gamma} |{\gamma}\in\mathbb{R}\alpha\oplus \mathbb{R}{\beta}\cap\Phi^+>=G_{u'}$, this is a consequence of the proof of Proposition \ref{prop1} due to the arbitrariness of $t_1$ and $t_2$.
\end{proof}

Using the partition of $B(2,e)$ defined through $\sim$ we are able to prove Theorem \ref{B(W)}. 

\begin{proof}[Proof of Theorem \ref{B(W)}]
First, we compute the Ricci curvature of the Bruhat graph of dihedral groups:

\textbf{Dihedral case.} Let $s$ and $s'$ be two reflections that generate $W=I_{2}(m)$. These
satisfy $s^{2}=e$, $s'^{2}=e$ and $(ss')^m=e$. The reflections in $W$ are $s_j=s(s's)^{j}$ for $j=0,\ldots,m-1$, notice that $s'=s_{m-1}$. The product of two reflections is a rotation, in $W$ there are $m$ rotations (including the trivial one), these are written as $\rho_i=(ss')^i$ for $i=1,\ldots,m$. We remark that given any reflection $s_j$ and any rotation $\rho_i$ there exists a reflection $s_k$ such that $s_js_k=\rho_{i}$, this was already noticed in the proof of Proposition \ref{prop1}.
The function on $W$  that acts by multiplication for an element $g\in W$ induces an automorphism of $B(W)$, by Lemma \ref{stesso} we see that $\ric(B(W))_e=\ric(B(W))_g$, so it is enough to compute $\ric(B(W))_e$.
 The elements in $B(1,e)$ are the $s_i$, $i=0,\ldots,n-1$, while the elements in $B(2,e)$ are the non trivial rotations in $W$.
Now we are ready for the computation. Let $u\in B(2,e)$ so $u=\rho_{\sigma}$ for some $\sigma\in \{1,\ldots,n\}$ and $f$ be a function defined on $W$ such that $f(e)=0$ and $\Gamma(f)(e)\neq 0$, we have that:

\begin{align*}
\sum_{v\in B(1,e)\cap B(1,u)}(f(u)-2f(v))^2&=
\sum_{i=0}^{n-1}(f(u)-2f(s_i))^2\\
&=nf(u)^2+\sum_{i=0}^{n-1}4f(s_i)^2-4f(u)\sum_{i=0}^{n-1}f(s_i)\\
&=(\sqrt{n}f(u)-\frac{2}{\sqrt{n}}\sum_{i=0}^{n-1}f(s_i))^2+\\
&\frac{4}{n}\sum_{0\leq
i<j\leq n-1}(f(s_i)-f(s_j))^2\\
  & \geq
\frac{4}{n}\sum_{0\leq i<j\leq n-1}
(f(s_i)-f(s_j))^2.
\end{align*}
Where we used the fact that $\sum_{0\leq i<j\leq n-1}(x_i^2+x_j^2)=(n-1)\sum_{i=0}^{n-1}x_i^2$.
Summing up for all the $u\in B(2,e)$ we have:
\begin{align*}
\frac{1}{2}\sum_{u\in B(2,e)}\sum_{v\in B(1,e)\cap B(1,u)}(f(u)-2f(v))^2 &\geq \frac{2(n-1)}{n}\sum_{0\leq i<j\leq n-1}
(f(s_i)-f(s_j))^2\\ &\geq \sum_{0\leq i<j\leq n-1}
(f(s_i)-f(s_j))^2.
\end{align*}
We use this to bound $\Gamma_2(f)(e)$:
\[
2\Gamma_2(f)(e)\geq\sum_{0\leq i<j\leq n-1}
(f(s_i)-f(s_j))^2+(\sum_{1\leq i\leq n}f(s_i))^2+\sum_{1\leq i\leq n}(2-d)f(s_i)^2
\]
where $d=d(e)=d(v)=n$.
Therefore we can conclude that $$2\Gamma_2(f)(e)\geq 4\sum_{0\leq i<j\leq n-1}(f(s_i)-f(s_j))^2=4\Gamma(f)(e).$$ From equation (\ref{equcurv}) we have that Ric$(B(W))\geq2$, thanks to Corollary $\ref{coro}$ the claim follows.

We now prove the result for any finite Coxeter group.

\textbf{General case.} By Lemma \ref{stesso} it is enough to compute $\ric(B(W))_e$.
Let $f:W\rightarrow \mathbb{R}$.
Let $U\subset B(2,e)$ be an equivalence class of the equivalence relation $\sim$ and let $G_U$ be the group associated to any element in $U$. We want to study
\[
\sum_{u\in U}\sum_{v\in B(1,e)\cap B(1,u)}(f(u)-2f(v))^2.
\]
First we notice that the elements in $B(1,u)\cap B(1,e)$ are the $v=t\in T$ such that $t't=u$ for some $t'\in T$, hence $B(1,u)\cap B(1,e)\subset G_U$ where $U$ is the class of $u$. With the same computation carried out to prove the dihedral case we conclude that:
\[
\sum_{u\in U}\sum_{v\in B(1,u)\cap B(1,e)}(f(u)-2f(v))^2\geq \sum_{v,v'\in B(1,e)\cap G_U} \frac{4(n_u-1)}{n_u}(f(v)-f(v'))^2.
\]
With $n_u=|B(1,u)\cap B(1,e)|$.

Summing up over all the equivalence classes $U$ we obtain
that
\begin{equation}\label{conto}
\sum_{u\in B(2,e)}\sum_{v\in B(1,e)\cap B(1,u)}\frac{1}{2}(f(u)-2f(v))^2\geq\sum_{v,v'\in B(1,e)}(f(v)-f(v'))^2.
\end{equation}
Here the last inequality follows from the remark that
any pair of reflections in $W$ is contained in (exactly) one group $G_U$ (Corollary \ref{coropair}).
From (\ref{conto}) we have that
\[
2\Gamma_2(f)(e)\geq\sum_{v\in B(1,e)}(f(v)-f(v'))^2+(\sum_{v\in B(1,e)})^2+\sum_{v\in B(1,e)}(2-d)f(v)^2.
\]
Therefore  $\Gamma_2(f)(e)\geq 2 \Gamma(f)(e)$ and $\Ric(G)\geq 2$. From Lemma \ref{coro} we conclude Ric$(B(W))=2$.
\end{proof}

\section{Applications}
In this Section we derive some applications of the result obtained in Section \ref{secbru}. More precisely we obtain a lower bound on the spectral gap of the Bruhat graphs of finite Coxeter groups and an isoperimetric inequality valid for any subset of finite Coxeter group.

\begin{corollary}\label{corospgap}
Let $W$ be a finite Coxeter group and $B(W)$ be its Bruhat graph. Then its spectral gap $\lambda \geq 2$.
\end{corollary}
\begin{proof}
This follows from Theorems \ref{B(W)} and \ref{spgap}.
\end{proof}

\begin{corollary}
Let $B(W)$ be the Bruhat graph associated to a finite Coxeter group $W$, let $A$ be a subset of $W$, then
\[
|\partial A|\geq \frac{1}{2}|A|\left(1-\frac{|A|}{|W|}\right).
\]
Where $\partial A$ denotes the set of edges in $B(W)$ that connects $A$ to its complement.
\end{corollary}
\begin{proof}
We know from Corollary \ref{corospgap} that  $\lambda_{B(W)}=\lambda\geq 2$. To apply Theorem \ref{sgiso} we notice that $min \{\sqrt{\lambda},\frac{\lambda}{\sqrt{2|K|}}\}\geq 1$. The statement follows. 
\end{proof}

\section*{Acknowledgement}
The author would like to thank Francesco Brenti for suggesting this problem and for many useful discussions. Furthermore she acknowledges the MIUR Excellence Department Project awarded to the Department of Mathematics, University of Rome Tor Vergata, CUP E83C18000100006.

\printbibliography

\end{document}